\documentclass[11pt]{amsart}

\usepackage{url}
\usepackage{hyperref}

\usepackage{amssymb,amsmath,amscd,graphicx,
latexsym,amsthm}
\usepackage{amssymb,latexsym,amsmath,amscd,graphicx}
\usepackage[mathscr]{eucal}
  \usepackage[all]{xy}
\usepackage[a4paper,top=3.5cm,bottom=3.5cm,left=3cm,right=3cm]{geometry}

\def\Z{{\mathbb Z}}
 
\def\Q{{\mathbb Q}}

\def\OO{{\mathcal O}}

\def\FF{\mathcal{F}}

\def\II{\mathcal{I}}

\def\l{\underline{{\textit{l}}}}

\def\tensor{{\otimes}}

\def\Pic0{\mathrm{Pic}^0}

\def\max{\mathrm{max}}

\def\la{{\langle}}
\def\ra{{\rangle}}

\theoremstyle{plain}

\newtheorem{theorem}{Theorem}[section]

\newtheorem{proposition/example}[theorem]{Proposition/Example}
\newtheorem{definition/theorem}[theorem]{Definition/Theorem}
\newtheorem{proposition}[theorem]{Proposition}

\newtheorem{corollary}[theorem]{Corollary}

\theoremstyle{definition}
\newtheorem{definition}[theorem]{Definition}

\newtheorem{remark}[theorem]{Remark}

\newtheorem{conjecture/question}[theorem]{Conjecture/Question}

\newtheorem{remark/definition}[theorem]{Remark/Definition}
\newtheorem{notation/assumptions}[theorem]{Assumptions/Notation}

\numberwithin{equation}{section}

 \theoremstyle{remark}

\begin{document}  

\title{
Higher order embeddings via
the basepoint-freeness threshold 
}

\author{Federico Caucci}
  \address{Dipartimento di Matematica e Informatica, Universit\`a di Ferrara, Via Machiavelli 30, 44121 Ferrara, Italy} 
 \email{{\tt federico.caucci@unife.it}}

\maketitle

\setlength{\parskip}{.1 in}

\begin{abstract} 
In this note, we relate the basepoint-freeness threshold of a polarized  abelian variety, introduced by 
Jiang and Pareschi, with  $k$-jet very ampleness. Then,
we derive several applications of this fact, including a criterion for the $k$-very ampleness of Kummer varieties.  
 \end{abstract}

\section{Introduction}  

Let $X$ be a complex projective variety. A line bundle $L$ on $X$ is said to be \emph{$k$-jet very ample}, $k \geq 0$, if the evaluation map
\[
H^0(X, L) \rightarrow H^0(X, L \otimes \OO_X / \II_{x_1}  \cdots \II_{x_{k+1}})
\]
is surjective for any, not necessarily distinct, $k+1$ points $x_1, \ldots, x_{k+1}$ in $X$. In particular, $0$-jet very ample means that $L$ is globally generated, and $1$-jet very ample means that $L$ is very ample. When $k \geq 2$, $k$-jet very ampleness is in general stronger than the related notion of \emph{$k$-very ampleness}, which takes into account $0$-dimensional subschemes $Z \subseteq X$ of length $k + 1$ asking that the evaluation maps
$$H^0(X, L) \rightarrow H^0(X, L \otimes \OO_Z)$$ 
are surjective.
 These and related conditions were introduced by Beltrametti,
Francia and Sommese in \cite{befrso}, as possible 
higher analogues of very ampleness.

In the first part of this short note, we show that, when $X = A$ is an abelian variety, the Jiang-Pareschi basepoint-freeness threshold 
 \cite{jipa} influences the higher order embeddings of $L$. Namely, we have:
\begin{theorem}\label{thm1}
Let $(A, \l\, )$ be a polarized abelian variety. If 
\begin{equation}\label{beta0}
\beta(A, \l\, ) < \frac{1}{k+1},
\end{equation}
then $L \otimes N$ is $k$-jet very ample, for any ample line bundle $L$ representing the polarization $\l$ and any nef line bundle $N$ on $A$.
\end{theorem}
\noindent We refer the reader to \S \ref{S2} for the definition and main properties of $\beta(A, \l\, )$. Theorem \ref{thm1} (partially) recovers the well-known result of Bauer-Szemberg \cite{basz} saying that the tensor power of $m$ ample line bundles on an abelian variety
 is $k$-jet very ample, when $m \geq k+2$ (see also \cite[Theorem 3.8]{papoII} for a more general result). 
See \S \ref{S2} for some other consequences which follow from Theorem \ref{thm1}.
Let us  just point out here that it can be applied to \emph{any} ample line bundle satisfying \eqref{beta0}, in particular to primitive ones.

In the latter part of the paper, 
we investigate higher order embeddings of a Kummer variety $K(A)$, that is the quotient  of an abelian variety $A$ by the action of the inverse morphism $A \ni p \mapsto -p$. It is known that an ample line bundle on $K(A)$ is globally generated, and that a square of it is very ample (indeed, it is projectively normal by Sasaki \cite{sa}).
We generalize these facts as follows.
\begin{theorem}\label{thm2}
If $L$ is an ample line bundle on $K(A)$ and $$m > \frac{k+1}{2},$$ then $L^{\otimes m} \otimes N$ is $k$-very ample, for any nef line bundle $N$ on $K(A)$.
\end{theorem} 
\noindent This complements 
our previous 
result  on syzygies of Kummer varieties \cite{caK}.

Proofs rely on some general criteria of Ein-Lazarsfeld-Yang \cite{einlazyang}, and Agostini \cite{ag}.
More precisely, in \cite{einlazyang, ag},  the authors,  with an approach built on \cite{eilagon} and \cite{yang} which in turn is inspired by the work of Voisin \cite{vo}, noted  that the \emph{asymptotic vanishing} of  certain cohomology groups  are related to the $k$-(jet) very ampleness of line bundles (see Theorems \ref{elyjet} and \ref{agocrit}). We produce sufficient conditions to get these asymptotic vanishings on abelian and Kummer varieties, so to derive information on their higher order embeddings.

\subsection*{Acknowledgements.  } Theorem \ref{thm1} is part of my Ph.D.\ thesis, supervised by Beppe Pareschi. I thank him for all his teachings.  I am also grateful  to  the referee for remarks improving Proposition \ref{posmult}.
 
\subsection*{Notations. }
We work (for simplicity) over the complex numbers. Let $X$ be a projective variety. If $L$ is an ample and globally generated line bundle on $X$, we denote by 
$M_L$ the kernel (or syzygy) bundle associated to $L$, which, by definition, sits in the short exact sequence $0 \to M_L \to H^0(X, L) \otimes \OO_X \to L \to 0$. 
 A polarized abelian variety $(A, \l\,)$ is the datum of an abelian variety $A$ and  a polarization $\l \in \mathrm{Pic} A / \Pic0 A$, which is the class of an ample line bundle on $A$. Given a non-zero integer $b \in \Z$, we denote by $\mu_b \colon A \to A$ the isogeny sending $p \in A$ to $b p \in A$.

\section{A criterion for the $k$-jet very ampleness of abelian varieties}\label{S2} Let $X$ be a  projective variety,  $L$ and $P$ be line 
bundles on $X$, with $L$ ample and $P$ arbitrary. 
We have the following criterion proved by Ein, Lazarsfeld and Yang:
\begin{theorem}[\cite{einlazyang}, Remark 1.8]\label{elyjet}
Assume that $X$ is smooth and
 $H^1(X, P)=0$. Then, $P$ is $k$-jet very ample  if and only if  the vanishing
\[
H^1(X, M_{L_d}^{\otimes (k+1)} \tensor P) = 0
\]
holds true for $d \gg 0$, where $L_d := L^{\tensor d}$.
\end{theorem}

Going back to the case of a polarized abelian variety $(A, \l\, )$, we want to use  Theorem \ref{elyjet}  to relate the following invariant, introduced in \cite{jipa},  with the notion of $k$-jet very  ampleness. 
Let us start by recalling a terminology introduced by Mukai \cite{mukai}. A coherent sheaf $\FF$  on $A$ is said to be $IT(0)$, if 
\[
H^i(A, \FF \otimes \alpha) = 0
\]
for all $i > 0$ and $\alpha \in \Pic0 A$.
\begin{definition}[\cite{jipa}]
The \emph{basepoint-freeness threshold} of a polarized abelian variety $(A, \l\, )$ is
\[
\beta(A, \l\, ) := \mathrm{Inf} \{r = \frac{a}{b} \in \Q^+ \ |\ \mu_b^* \II_0 \otimes L^{\otimes ab}\ \textrm{is\ $IT(0)$} \},
\]
where $\II_0 \subseteq \OO_A$ is the ideal sheaf of $0 \in A$ and $L$ represents the class $\l$.\footnote{It is not difficult to see that the definition does not depend on the represantation $r = \frac{a}{b}$.}
\end{definition}
 Jiang and Pareschi (see \cite[p.\ 842]{jipa}) showed that $\beta(A, \l\, ) \leq 1$, and one has the strict inequality if and only if $L$ is globally generated. Moreover, $\beta(A, m \l\, ) = \frac{\beta(A,\, \l\, )}{m}$ for all integers $m \geq 1$.

\begin{proof}[Proof of Theorem \ref{thm1}]
Let $L$ be an ample line bundle on $A$ which represents $\l$. Note that $P := L \tensor N$ is  ample, hence $H^1(A, P)=0$. 
By using the $\Q$-twisted notation (see \cite[Remark 2.2]{jipa}), for any $d \gg 0$ (actually, $d \geq 2$ suffices) we may write
\begin{equation}\label{ratwriting}
M_{L_d}^{\otimes (k+1)} \tensor\, P = M_{L_d}^{\otimes (k+1)} \tensor\, L\, \tensor\, N 
                        = \bigl(M_{L_d} \la \frac{1}{d(k+1)} \l_d \ra \bigr)^{\otimes (k+1)} \tensor\, N,
\end{equation}
where $\l_d := d\l$ denotes the class of $L_d$.
By \cite[Theorem D]{jipa} (see also \cite[Lemma 3.3]{ca}), the $\Q$-twisted sheaf $M_{L_d} \la \frac{1}{d(k+1)} \l_d \ra$ is $IT(0)$ (i.e., $\mu_{d(k+1)}^*M_{L_d} \otimes L_d^{\otimes d(k+1)}$ is so\footnote{The reason bringing to this definition is that, if $\FF\la \frac{a}{b} \l\, \ra$ is a $\Q$-twisted sheaf, then
$\mu_b^*(\FF\la \frac{a}{b} \l\, \ra) = \mu_b^* \FF \la \frac{a}{b} \mu_b^* \l\, \ra = \mu_b^* \FF \la \frac{a}{b} \cdot b^2 \l\, \ra$,  
and the last term is an actual sheaf.}) if and only if  
\begin{equation}\label{k1}
\frac{\beta(A, \l_d )}{1 - \beta(A, \l_d )} = \frac{\beta(A, \l\, )}{d-\beta(A, \l\, )} < \frac{1}{d(k+1)}.
\end{equation}
 Since \eqref{k1} is equivalent to
\begin{equation}\label{k2}
\beta(A, \l\, ) < \frac{d}{1+d(k+1)}
\end{equation} 
and the right-hand side of \eqref{k2} grows to $\frac{1}{k+1}$ as $d$ goes to $+\infty$, and  since we are assuming that $\beta(A, \l\, ) < \frac{1}{k+1}$, it is certainly possible to take $d$ big enough such that \eqref{k2} is satisfied. Therefore,   $M_{L_d} \la \frac{1}{d(k+1)} \l_d \ra$ is $IT(0)$ if $d \gg 0$.

In order to conclude the proof, we just need to observe that,
when $d \gg 0$,   in \eqref{ratwriting}
  we have written  the sheaf $M_{L_d}^{\otimes (k+1)} \tensor\, P$ as a tensor product of  $\Q$-twisted $IT(0)$ sheaves, tensored with a nef line bundle. So,  $M_{L_d}^{\otimes (k+1)} \tensor\, P$  is  $IT(0)$ by the preservation of vanishing principle (see  \cite[Proposition 3.4]{ca} which is based on \cite[Proposition 3.1]{papoIII}\footnote{Note that, since $N$ is a nef line bundle on an abelian variety,  $N\la x\l\,\ra$ is $IT(0)$ for any $x \in \Q^+$, that is, $N$ is a $GV$ (generic vanishing) sheaf. This is not the usual definition of a $GV$ sheaf, but it is equivalent to it by \cite[Theorem 5.2(a)]{jipa}.}).  
This gives
$H^1(A, M_{L_d}^{\otimes (k+1)} \tensor P) = 0$
for $d \gg 0$. So, Ein-Lazarsfeld-Yang characterization (Theorem \ref{elyjet}) implies that the line bundle $P$ is $k$-jet very ample. 
\end{proof}
\begin{remark}
The above proof was contained in the author's Phd thesis. More recently,  
Ito  found a slightly finer result, with a different technique, when $k \geq 1$ (see \cite[Theorem 1.6]{itoM}).
\end{remark}

\begin{corollary}\label{corlocpos}
If $\beta(A, \l\, ) < \frac{m}{k+1}$, then $L^{\otimes m}$ is $k$-jet very ample. In particular, $L^{\otimes m}$ is $k$-jet very ample whenever $m \geq k + 2$. 
\end{corollary}
\begin{proof}
Since
$\beta(A, m\l\, ) = \frac{\beta(A,\, \l\, )}{m} < \frac{1}{k+1}$,
$L^{\otimes m}$ is $k$-jet very ample by Theorem \ref{thm1}. The last statement follows as, in any case, $\beta(A, \l\, ) \leq 1$  holds true.
\end{proof}

Defining $t(\l\,) := \max \{ t \in \mathbb{N} \ |\ \beta(A, \l\, ) \leq \frac{1}{t} \}$,
we have, more in general,  the following 
\begin{proposition}\label{posmult}
Let $k \geq 0, m \geq 2$ be integers, and  $L$ be a basepoint-free line bundle on $A$. Then, $L^{\otimes m}$ is $k$-jet very ample if $m \geq k+2-t(\l\,)$.
\end{proposition}
\begin{proof}
If $t(\l\,) = 1$, then $m \geq k+1$ and, hence, the statement follows from Corollary \ref{corlocpos} as $L$ is basepoint-free. 
Let us assume $t(\l\,) > 1$. 
By Theorem \ref{thm1}, it suffices to show that $\beta(A, m\l\, ) < \frac{1}{k+1}$. If $k \geq t(\l\,)$, 
then
\begin{equation*}
\beta(A, m\l\, ) = \frac{\beta(A, \l\, )}{m} \leq \frac{1}{m \cdot t(\l\,)}	\leq \frac{1}{(k+2-t(\l\,)) t(\l\,)} < \frac{1}{k+1},
\end{equation*} 
where the last inequality is equivalent to
$t(\l\,)^2 - (k+2)t(\l\,) + k+1 < 0$ which  is satisfied if and only if $1 < t(\l\,) < k + 1$. So, we have done in this case.
If $t(\l\,) \geq k+1$, then 
\[
\beta(A, m\l\, ) = \frac{\beta(A, \l\, )}{m} \leq \frac{1}{m \cdot t(\l\,)}	\leq \frac{1}{2(k+1)} < \frac{1}{k+1},
\]
as $m \geq 2$.
\end{proof}

The main result of \cite{ca} established a  relation (formally similar to the one stated in Theorem \ref{thm1}) between $\beta(A, \l\, )$ and the syzygies of $(A, \l\,)$. Based on this fact, a certain interest arises in bounding from above the basepoint-freeness threshold. Some quite general results were obtained, in particular, by Ito \cite{ito3} and Jiang \cite{ji}. What follows is a direct consequence of these and our result above.

\begin{proposition}
Let $(A, \l\, )$ be a polarized abelian variety. Let $L$ be an ample line bundle representing $\l\,$. 
If:

\noindent (i) $\dim A \leq 3$ and $(L^{\dim B} \cdot B) > ((k+1)\dim B)^{\dim B}$ for any non-zero abelian subvariety $B \subseteq A$;

or

\noindent (ii) $\dim A \geq  4$ and $(L^{\dim B} \cdot B) > (2 (k+1) \dim B)^{\dim B}$ for any non-zero abelian subvariety $B \subseteq A$,

\noindent then $L$ is $k$-jet very ample.
\end{proposition}
\begin{proof}
If $\dim A \leq 2$ (resp.\ $\dim A = 3$), the first set of inequalities implies that $\beta(A, \l\, ) < \frac{1}{k+1}$ by \cite[p.\ 950]{ca} (resp.\ \cite[Theorem 1.2]{ito3}). Even if it is expected that the same holds in any dimension, at the moment of writing the best general result is due to Jiang \cite[Theorem 1.4]{ji} who, using some birational geometry methods, proved that  $\beta(A, \l\, ) < \frac{1}{k+1}$, if the stronger inequalities in $(ii)$ are satisfied. Then, it suffices to apply Theorem \ref{thm1}.
\end{proof}

\section{Higher order embeddings of Kummer varieties}

Ein-Lazarsfeld-Yang result concerns smooth projective varieties (see however \cite[Remark 1.3]{einlazyang}).
For  possibly singular varieties, Agostini recently proved a similar -- although  weaker -- criterion  (see \cite[Theorem A and Proposition 2.2]{ag}):
\begin{theorem}[\cite{ag}]\label{agocrit}
Let $X$ be a projective variety, and $P$ be a line bundle on $X$. If 
\[
H^1(X, M_{L_d}^{\otimes (k+1)} \otimes P) = 0
\]
for $d \gg 0$,
 then $P$ is $k$-very ample.
\end{theorem}

We plan to apply it to the Kummer variety $K(A)$ associated to an abelian variety $A$. More precisely,
if $m$ is an integer and $L$ (resp.\ $N$) is an ample (resp.\ nef) line bundle on $K(A)$, we aim to prove that $H^1(K(A), M_{L_d}^{\otimes (k+1)} \otimes L^{\otimes m} \otimes N) = 0$, if  $m > \frac{k+1}{2}$ and $d \gg 0$.
\begin{proof}[Proof of Theorem \ref{thm2}]
Let $\pi \colon A \to K(A)$ be the quotient morphism. Since $\OO_{K(A)}$ is a direct summand of $\pi_*\OO_A$, one has 
\[
H^1(K(A), M_{L_d}^{\otimes (k+1)} \otimes L^{\otimes m} \otimes N) \subseteq H^1(A, \pi^*(M_{L_d}^{\otimes (k+1)} \otimes L^{\otimes m} \otimes N)),
\]
where in the right-hand side we used the projection formula and the fact that $\pi$ is a finite morphism.
 Then, it suffices to prove that $H^1(A, \pi^*(M_{L_d}^{\otimes (k+1)} \otimes L^{\otimes m} \otimes N)) = 0$ when $m > \frac{k+1}{2}$ and $d \gg 0$.
Let us write
\begin{equation*}
\begin{split}
\pi^*(M_{L_d}^{\otimes (k+1)} \otimes L^{\otimes m} \otimes N) &= \big( \pi^*M_{L_d} \la \frac{m}{k+1} \pi^*\l\, \ra \big)^{\otimes (k+1)} \otimes \pi^* N \\
&= \big( \pi^*M_{L_d} \la \frac{m}{d(k+1)} \pi^*\l_d \ra \big)^{\otimes (k+1)} \otimes \pi^* N.
\end{split}
\end{equation*}
From \cite[Proof of Lemma 5.4]{caK}, we know 
that the $\Q$-twisted sheaf $\pi^*M_{L_d} \la \frac{m}{d(k+1)} \pi^* \l_{d} \ra$ is $IT(0)$, if the inequality
\[
\frac{1}{2d} < \frac{\frac{m}{d(k+1)}}{1+ \frac{m}{d(k+1)}} = \frac{m}{d(k+1) + m}
\] 
holds true, i.e., if $(k+1)d + m < 2md$. When $d \gg 0$, this is the case if $2m > k+1$, which is precisely what we are assuming by the hypothesis.
Then, since $\pi^* N$ is a nef line bundle on $A$, we just need to apply, as before, the preservation of vanishing to get that $\pi^*(M_{L_d}^{\otimes (k+1)} \otimes L^{\otimes m} \otimes N)$ is an $IT(0)$ sheaf.
\end{proof}

\providecommand{\bysame}{\leavevmode\hbox
to3em{\hrulefill}\thinspace}

\end{document}